\documentclass[runningheads]{llncs}
\usepackage{graphicx}
\usepackage{graphicx}
\usepackage{makeidx}
\usepackage{graphics}
\graphicspath{ {images/} }
\usepackage{amsmath}
\usepackage{amsfonts}
\usepackage{amssymb}
\usepackage{epsfig}
\usepackage{tabularx}
\usepackage{latexsym}
\usepackage{mathtools}
\usepackage{epstopdf}
\usepackage{array}
\usepackage[utf8]{inputenc}
\usepackage{caption}
\usepackage{enumerate}
\usepackage{mleftright}
\usepackage{mathpazo}
\newcommand\test[2]{\mleft(\knds\genfrac..{0pt}{}{#1}{#2}\knds\mright)}
\newcommand{\knds}{\kern-\nulldelimiterspace}

\newtheorem{observation}{Observation}

\newtheorem{case*}{Case}
\newtheorem{case**}{Case}

\newcommand{\floor}[1]{\left\lfloor #1\right \rfloor}
\newcommand{\ceil}[1]{\left\lceil #1 \right\rceil}

\begin{document}
\title{$k$-Sets and Rectilinear Crossings in Complete Uniform Hypergraphs}
%
%
\author{Rahul Gangopadhyay\inst{1} \and
Saswata Shannigrahi\inst{2}}
\authorrunning{R. Gangopadhyay et al.}
%
\institute{ IIIT Delhi, India\\
\email{rahulg@iiitd.ac.in}\\
\and
Saint Petersburg State University, Russia\\
\email{saswata.shannigrahi@gmail.com}}
\maketitle              
\begin{abstract}
 In this paper, we study the $d$-dimensional rectilinear drawings of the complete $d$-uniform hypergraph $K_{2d}^d$. Anshu et al. [Computational Geometry: Theory and Applications, 2017] used Gale transform and Ham-Sandwich theorem to prove that there exist  $\Omega \left(2^d\right)$ crossing pairs of hyperedges in such a drawing of $K_{2d}^d$.  We improve this lower bound by showing that there exist  $\Omega \left(2^d \sqrt{ d}\right)$ crossing pairs of hyperedges in a $d$-dimensional rectilinear drawing of $K_{2d}^d$. We also prove the following results. \\
 
1.  There are $\Omega \left(2^d {d^{3/2}}\right)$ crossing pairs of hyperedges in a $d$-dimensional rectilinear drawing of $K_{2d}^d$ when its $2d$ vertices are either not in convex position in $\mathbb{R}^d$ or form the vertices of a $d$-dimensional convex polytope that is $t$-neighborly but not $(t+1)$-neighborly for some constant $t\geq1$ independent of $d$.\\

2. There are $\Omega \left(2^d {d^{5/2}}\right)$ crossing pairs of hyperedges in a $d$-dimensional rectilinear drawing of $K_{2d}^d$ when its $2d$ vertices form the vertices of a $d$-dimensional convex polytope that is $(\floor{d/2}-t')$-neighborly for some constant $t' \geq 0$ independent of $d$.  

  
\end{abstract}
\vspace{-0.5cm}
\section{Introduction}
\label{Intro}
For a sufficiently large positive integer $d$, let $K_{n}^d$ denote the complete $d$-uniform hypergraph with $n \geq 2d$ vertices. A \textit{$d$-dimensional rectilinear drawing} of $K_{n}^d$ is defined as an embedding of it in $\mathbb{R}^d$ such that the vertices of $K_{n}^d$ are points in general position in $\mathbb{R}^d$ and each hyperedge is drawn as the $(d-1)$-simplex formed by the $d$ corresponding vertices \cite{AS}. Note that a set of points in $\mathbb{R}^d$ is said to be in \textit{general position} if no $d+1$ of them lie on a $(d-1)$-dimensional hyperplane. For $u$ and $v$ in the range $0 \leq u, v \leq d-1$, a $u$-simplex spanned by
a set of $u+1$ points and a $v$-simplex spanned by a set of $v+1$ points (when these $u+v+2$ points are in general position in $\mathbb{R}^d$) are said to be {\textit {crossing}} if they do not have common vertices ($0$-faces) and contain a common point in their relative interiors \cite{DP}.  As a result, a pair of hyperedges in a $d$-dimensional rectilinear drawing of  $K_{n}^d$ are said to be \textit{crossing} if they do not have a common vertex  and  contain a common point in their relative interiors \cite{AGS,AS,DP}. The {\textit{$d$-dimensional rectilinear crossing number}} of $K_{n}^d$, denoted by $\overline {cr}_d(K_{n}^d)$, is defined as the minimum number of crossing pairs of hyperedges among all $d$-dimensional rectilinear drawings of $K_{n}^d$ \cite{AGS,AS}. \par
Since the crossing pairs of hyperedges formed by a set containing $2d$ vertices of $K_{n}^d$  are distinct from the crossing pairs of hyperedges formed by another set of $2d$ vertices, it can be observed that $\overline {cr}_d(K_{n}^d) \geq \overline {cr}_d(K_{2d}^d)\dbinom{n}{2d}$ \cite{AS}. The best-known lower bound on $\overline {cr}_d(K_{2d}^d)$ is $\Omega (2^d)$ \cite{AGS}.  Anshu et al. \cite{AGS} also studied a $d$-dimensional rectilinear drawing of $K_{2d}^d$ where all its vertices are placed on the {\textit{$d$-dimensional moment curve}} $\gamma=\{(a,a^2,\ldots, a^d): a \in \mathbb{R}\}$ and proved that the number of crossing pairs of hyperedges is  $\Theta\left(4^d/ \sqrt{d}\right)$ in this drawing. \par
As described above, an improvement of the lower bound on  $\overline {cr}_d(K_{2d}^d)$  improves the lower bound on  $\overline {cr}_d(K_{n}^d)$.
Let us denote the points corresponding to the set of vertices in a $d$-dimensional rectilinear drawing of the hypergraph $K_{2d}^d$ by 
$V = \{v_1, v_2, \ldots, v_{2d}\}$.  
The points in $V$ are said to be in  {\textit{convex position}} if there does not exist any point $v_i \in V$ (for some  $i$ in the range $1\leq i \leq 2d$) such that $v_i$ can be expressed as a convex combination of the points in $V \setminus \{v_i\}$, and such a drawing of $K_{2d}^d$ is called a {\textit{$d$-dimensional convex drawing}} of it \cite{AGS}.  \par
 Note that the convex hull of the vertices of $K_{2d}^d$ in a $d$-dimensional convex drawing of it forms a $\textit{$d$-dimensional convex polytope}$ with its vertices in general position.  For any $t\geq 1$, a \textit{$d$-dimensional $t$-neighborly polytope} is a 
$d$-dimensional convex polytope in which each subset of its vertex set having less than or equal to $t$ vertices forms a face \cite[Page 122]{GRU}.  Such a polytope is said to be  \textit{neighborly}  if it is $\floor{d/2}$-neighborly. Note that any $d$-dimensional convex polytope can be at most $\floor{d/2}$-neighborly unless it is a $d$-simplex which is $d$-neighborly \cite[Page 123]{GRU}. The {\textit{$d$-dimensional cyclic polytope}} is a special kind of neighborly polytope where all its vertices are placed on the $d$-dimensional moment curve \cite[Page 15]{ZIE}. Using these definitions and notations, let us describe our contributions in this paper.

In Section \ref{imprvcross}, we improve the lower bound on $\overline {cr}_d(K_{2d}^d)$. In particular, we prove the following theorem which implies Corollary \ref{thm41}.

\begin{theorem}
\label{thm4}
$\overline {cr}_d(K_{2d}^d)= \Omega(2^d \sqrt{d})$.
\end{theorem}
\begin{corollary}
\label{thm41}
$\overline {cr}_d(K_{n}^d)= \Omega(2^d \sqrt{d})\dbinom{n}{2d}$.
\end{corollary}
We then prove the following theorems to obtain lower bounds on the number of crossing pairs of hyperedges in some special types of $d$-dimensional rectilinear drawings of  $K_{2d}^d$.
 \begin{theorem}
\label{thm1}
The number of crossing pairs of hyperedges in a $d$-dimensional rectilinear drawing of $K_{2d}^d$ is $\Omega (2^d {d^{3/2}})$ if the vertices of $K_{2d}^d$ are not in convex position.\end{theorem}
\begin{theorem}
\label{thm2}
For any constant $t \geq 1$ independent of $d$, the number of crossing pairs of hyperedges in a $d$-dimensional rectilinear drawing of $K_{2d}^d$ is $\Omega (2^d {d^{3/2}})$ if the vertices of $K_{2d}^d$ are placed as the vertices of a $d$-dimensional $t$-neighborly polytope that is not $(t+1)$-neighborly.
\end{theorem}
\begin{theorem}
\label{thm3}
For any constant $t' \geq 0$ independent of $d$, the number of crossing pairs of hyperedges in a $d$-dimensional rectilinear drawing of $K_{2d}^d$ is $\Omega (2^d d^{5/2})$ if the vertices of $K_{2d}^d$ are placed as the vertices of a $d$-dimensional $(\floor{d/2}-t')$-neighborly polytope. 
\end{theorem}

\noindent Note that the $t'=0$ case in Theorem \ref{thm3} corresponds to a neighborly polytope. As mentioned above, the number of crossing pairs of hyperedges in a $d$-dimensional rectilinear drawing of $K_{2d}^d$ is known to be $\Theta(4^d/\sqrt{d})$ when such a polytope is cyclic.\\

\noindent {\bf Techniques Used:} We use the properties of Gale transform, affine Gale diagram, $k$-sets and balanced lines to prove Theorems 
\ref{thm4}, \ref{thm1}, \ref{thm2} and \ref{thm3}. We discuss these properties in detail in Sections \ref{TU} and \ref{TU1}. In addition, a few other results used in these proofs are mentioned below.
Let us first state  the Ham-Sandwich theorem and the Caratheodory's theorem  that are used in the proofs of Theorems \ref{thm4} and \ref{thm1}, respectively.\\
 
 \noindent{\textbf {Ham-Sandwich Theorem:}} \cite{JM,STO}  There exists a $(d-1)$-hyperplane $h$ which simultaneously bisects $d$ 
finite point sets $P_1, P_2, \ldots, P_d$ in $\mathbb{R}^d$, such that each of 
the 
open half-spaces created by $h$ contains at most 
$\left\lfloor\small{{|P_i|}/{2}}\right\rfloor$ 
points of $P_i$ for each $i$ in the range $1 \leq i \leq d$.\\

\noindent{\textbf {Carath\'{e}odory's Theorem:}} \cite{JM,ST} Let $X \subseteq \mathbb{R}^d$. Then, each point in the convex hull $Conv(X)$ of $X$ can be expressed as a convex combination of at most $d+1$ points in $X$.\\

\noindent In the following, we state the Proper Separation theorem that is used in the proof of Lemma \ref{extension}. Two non-empty convex sets are said to be \textit{properly separated} in $\mathbb{R}^d$ if they lie in the opposite closed half-spaces created by a $(d-1)$-dimensional hyperplane and both of them are not contained in the hyperplane.  \\

\noindent{\textbf {Proper Separation Theorem:}} \cite[Page 148]{OG} Two non-empty convex sets can be properly separated in $\mathbb{R}^d$ if and only if  their relative interiors are disjoint.\\

\noindent The proof of the following lemma, which is used in the proofs of all four theorems, is the same proof mentioned in \cite{AGS} for the special case $u=v=d-1$. For the sake of completeness, we repeat its proof in full generality.
\begin{lemma}\cite{AGS}
\label{extension} 
Consider a set $A$ that contains at least $d+1$ points in general position in $\mathbb{R}^d$. Let $B$ and $C$ be its disjoint subsets such that $|B|= b$, $|C|= 
c$, $2 \leq b,c \leq d$ and $b+c \geq d+1$. If the $(b-1)$-simplex formed by 
$B$ and the
$(c-1)$-simplex formed by $C$ form a crossing pair, then the $u$-simplex ($u \geq b-1$) formed by a point set $B'\supseteq B$ and  the $v$-simplex ($v \geq c-1$) formed by a point set $C' \supseteq C$ satisfying $B' \cap C' = \emptyset$, $|B'|, |C'| \leq d$  and $B', C' \subset A$ also form a crossing pair.
\end{lemma}
 \begin{proof}
For the sake of contradiction, we assume that there exist a $u$-simplex and a $v$-simplex, formed respectively by the disjoint point sets $B' \supseteq B$ and $C' \supseteq C$, that do not cross. We consider two cases.
\begin{case**}
{\normalfont Let us assume that $Conv(B') \cap Conv(C') = \emptyset$. It clearly leads to a contradiction since $Conv(B) \cap Conv(C) \ne \emptyset$.}
\end{case**}
\begin{case**}
{\normalfont Let us assume that $Conv(B') \cap Conv (C') \neq \emptyset$. Since the relative interiors of $Conv(B')$ and $Conv(C')$ are disjoint, the Proper Separation theorem implies that there exists a $(d-1)$-dimensional hyperplane $h$ such that $Conv(B')$ and $Conv(C')$ lie in the opposite closed half-spaces determined by $h$.   It implies that $Conv(B)$ and $Conv(C)$ also lie in the opposite closed half-spaces created by $h$. Since the relative interiors of $Conv(B)$ and $Conv(C)$ are not disjoint and they lie in the opposite closed halfspaces of $h$, it implies that all $b+c \geq d+1$ points in $B \cup C$ lie on $h$.  This leads to a contradiction since the points in $B \cup C$ are in general position in $\mathbb{R}^d$.} \qed
\end{case**}
\end{proof}
\vspace{-0.5cm}
\section{Gale Transform and its Properties}
\label{TU}
The {\textit{Gale transformation}} is a useful technique to investigate the properties of high dimensional point sets \cite{GL}. Consider a sequence of $m>d+1$ points $P=$ $<p_1,p_2, \ldots, p_m>$ in $\mathbb{R}^d$ such that the affine hull of the points is $\mathbb{R}^d$. Let the  $i^{th}$ point $p_i$ be represented as $(x_1^i, x_2^i, \ldots, x_d^i)$. To compute a {\textit{Gale transform}}  of $P$, let us consider the $(d+1)\times m$ matrix $M(P)$ whose $i^{th}$ column is 
$\begin{pmatrix}
  x_1^i \\
  x_2^i \\
  \vdots \\
  x_d^i \\
  1
\end{pmatrix}
$.
Since there exists a set of $d+1$ points in $P$ that is affinely independent, the rank of  $M(P)$ is $d+1$. Therefore, the dimension of the null space of $M(P)$ is $m-d-1$.  Let $\{(b_1^1, b_2^1, \ldots, b_m^1),$  $(b_1^2, b_2^2, \ldots, 
b_m^2), \ldots,(b_1^{m-d-1}, b_2^{m-d-1},$ $\ldots, b_m^{m-d-1}) \}$ be a set of $m-d-1$ vectors that spans the null space of $M(P)$. A Gale transform $D(P)$ is the sequence of vectors $D(P)$ $=$ $<g_1,g_2, \ldots, g_m>$ where $g_i= (b_i^1$, $b_i^2$, $\ldots, 
b_i^{m-d-1})$ for each  $i$ satisfying $1 \leq i \leq m$. Note that $D(P)$ can also be treated as a point sequence in $\mathbb{R}^{m-d-1}$.\par 
We define a {\textit{linear separation}}  of $D (P)$ to be a partition of the vectors in $D(P)$  into two disjoint sets of vectors $D^+(P)$ and $D^-(P)$ contained in the opposite open half-spaces created by a linear hyperplane (i.e., a hyperplane passing through the origin). A linear separation is said to be \textit{proper} if one of the sets among  $D^+(P)$ and $D^-(P)$ contains $\floor{{m}/{2}}$ vectors and the other contains $\ceil{{m}/{2}}$ vectors.  We use the following properties of $D(P)$ in the proofs of our theorems. The first two of these properties are used in the proofs of all four theorems. The third property is used in the proof of Theorem \ref{thm1} and the last one is used in the proof of Theorem \ref{thm2}.

\begin{lemma}\cite[Page 111]{JM}
\label{genposi}
If the points 
in $P$ are in general position in $\mathbb{R}^d$, each collection of $m-d-1$ vectors in $D(P)$ spans $\mathbb{R}^{m-d-1}$. 
\end{lemma}

\begin{lemma}\cite[Page 111]{JM}
\label{bjection}
Consider two integers $u$ and $v$ satisfying  $1 \leq u, v \leq d-1$ and $u+v+2 = m$. If the points in $P$ are in general 
position in $\mathbb{R}^d$, there exists a 
bijection between the crossing pairs of $u$- and $v$-simplices formed by some points
in $P$ and the linear separations of  $D(P)$ into $D^+(P)$ and $D^-(P)$ 
such that $|D^+(P)| = u+1$ and $|D^-(P)| = v+1$.  
\end{lemma}
\begin{lemma}\cite[Page 111]{JM}
\label{convexity}
The points in $P$ are in convex position in $\mathbb{R}^d$ if 
and only if there 
is 
no linear hyperplane  with exactly one vector from $D(P)$ in one of the open half-spaces created by it.
\end{lemma}
\begin{lemma}\cite[Page 126]{GRU}
\label{neigh}
If the points in $P$ are in convex position, the $d$-dimensional polytope formed by the convex hull of the points in $P$ is $t$-neighborly if and only if each of the open half-spaces created by a linear hyperplane  contains at least $t+1$ vectors of $D(P)$.
\end{lemma}
The proofs of the first three lemmas are easy to see. The fourth lemma, a generalized version of which is given as an exercise in \cite{GRU}, can be proved using Lemma \ref{bjection} and the fact \cite[Page 111]{JM} that any $t$-element subset $P'= \{p_{i_1},p_{i_2}, \ldots,$ $ p_{i_t}\}  \subset P$ forms a $(t-1)$-dimensional face of the  $d$-dimensional polytope formed by the convex hull of the  points in $P$ if and only if the convex hull of the points in $D(P)\setminus \{g_{i_1},g_{i_2}, \ldots, g_{i_t}\}$ contains the origin.

We obtain an {\textit{affine Gale diagram}} \cite[Page 112]{JM} of $P$ by considering a hyperplane $\bar{h}$ that is not parallel to any vector in $D(P)$ and  not passing through the origin. Since the points in $P$ are in general position in $\mathbb{R}^d$, Lemma \ref{genposi} implies that $g_i \neq 0$ for every $i$. For each $i$ in the range $1\leq i \leq m$, we extend the vector $g_i \in D(P)$ either in the direction of $g_i$ or in its opposite direction until it cuts $\bar{h}$ at the point $\overline {g_i}$.  We color $\overline {g_i}$ as \textit{white} if the projection is in the direction of $g_i$, and \textit{black} otherwise. The sequence of $m$ points $\overline{D(P)}$ $=$ $< \overline {g_1}, \overline {g_2}, \ldots, \overline {g_m}>$ in $\mathbb{R}^{m-d-2}$ along with the color of each point is defined as an affine Gale diagram of $P$. We define a {\textit {partition}} of the points in $\overline{D(P)}$ as two disjoint sets of points $\overline{D^+(P)}$ and $\overline{D^-(P)}$ contained in the opposite open half-spaces created by a hyperplane. Let us restate  Lemma \ref{bjection} using these definitions and notations.
\begin{lemma}\cite{JM}
\label{bjection1}
Consider two integers  $u$ and $v$ satisfying  $1 \leq u, v \leq d-1$ and $u+v+2 = m$. 
If the points in $P$ are in general 
position in $\mathbb{R}^d$, there exists a 
bijection between the crossing pairs of $u$- and $v$-simplices formed by some points
in $P$ and the partitions of the points in $\overline{D(P)}$ into $\overline{D^+(P)}$ and $\overline{D^-(P)}$ 
such that $the ~number~ of~ white$ $points~in~\overline{D^+(P)}$ $plus~the ~number$ $~ of~ black~ points~in$ $\overline{D^-(P)}$ $~is~u+1$ and $the ~number~$ $of~ white~ points$ $~in~\overline{D^-(P)}~$ $plus$ $the ~number~ of~ black~ points$ $in~\overline{D^+(P)}~is~ v+1$.
\end{lemma}

\section{Balanced Lines, $j$-Facets and $k$-Sets }
\label{TU1}
In this section, we describe the properties of balanced lines, $j$-facets and $k$-sets that are used in the proofs of all four theorems.

\vspace{-0.2cm}
\subsection{Balanced Lines}
\noindent Consider a set $R$ containing $r$ points in general position in $\mathbb{R}^2$, such that $\ceil{{r}/{2}}$ points are colored white and $\floor{{r}/{2}}$ points are colored black. Let us state the definitions of a {\textit {balanced line}} and an {\textit {almost balanced line}} of $R$, and discuss their properties that are used in the proof of Theorem \ref{thm4}.\\

\noindent\textbf{Balanced Line:} \cite{PP}  A balanced line $l$ of $R$ is a straight line that passes through a white and a black point in $R$ and the number of black points is equal to the number of white points in each of the open half-spaces created by $l$.\\

\noindent Note that a balanced line exists only when $r$ is even. The following lemma gives a non-trivial lower bound on the number of balanced lines of $R$.

\begin{lemma}\cite{PP}
\label{bline}
When $r$ is even, the number of balanced lines of $R$  is at least $r/2$.
\end{lemma}

\noindent We extend the definition of a balanced line to define an \textit{almost balanced directed line} of $R$.\\

\noindent\textbf{Almost Balanced Directed Line:}  When $r$ is even, an almost balanced directed line $l$ of $R$ is a balanced line with direction assigned from the
black point to the white point it passes through. When $r$ is odd, an almost balanced directed line $l$ of $R$ is a directed straight line that passes through a white and a black point in $R$ such that the number of black points is equal to the number of white points in  the positive open half-space created by $l$.\\

\noindent We obtain the following observation from Lemma \ref{bline}.

\begin{observation} 
\label{obs:222}
The number of almost balanced directed lines of $R$  is at least $\floor{{r}/{2}}$.
\end{observation}
 
\vspace{-0.3cm}

\subsection{$j$-Facets and  $k$-Sets}
Consider a set $S$ containing $s$ points in general position in $\mathbb{R}^3$. Let us first state the definitions of  a \textit{$j$-facet} and an \textit{$(\leq j)$-facet} of $S$ for some integer $j \geq 0$. We then state the definitions of a \textit{$k$-set} and an \textit{$(\leq k)$-set} of $S$ for some integer $k \geq 1$, and discuss their properties that are used in the proofs of Theorems \ref{thm1}, \ref{thm2} and \ref{thm3}.  \\ 

\noindent{\textbf {$j$-facet:}} \cite{ANDR}  A {\textit{$j$-facet}} of $S$ is an oriented $2$-dimensional hyperplane spanned by $3$ points in $S$ such that exactly $j$ points of $S$ lie in the positive open half-space created by it.
\vspace{0.2cm}\\
\noindent  Let us denote the number of $j$-facets of $S$ by $E_{j}$.
\vspace{0.2cm}\\
\noindent{\textbf {$(\leq j)$-facet:}} \cite{ANDR}   An $(\leq j)$-facet of $S$ is an oriented $2$-dimensional hyperplane $h$ spanned by $3$ points in $S$ such that at most $j$ points of $S$ lie in the positive open half-space created by it.\\

\noindent{\textbf {Almost Halving Triangle:}}  An almost halving triangle of $S$ is a $j$-facet of $S$ such that $\left| j - (s-j-3)\right|$ is at most one.\\

\noindent When $s$ is odd, note that an almost halving triangle is a {\textit{halving triangle}} containing an equal number of points in each of the open half-spaces
created by it. The following lemma gives a non-trivial lower bound on the number of halving triangles of $S$. In fact, it is shown in \cite{MS} that this lemma is equivalent to Lemma \ref{bline}.  

\begin{lemma}\cite{MS}
\label{HTR}
When $s$ is odd, the number of halving triangles of $S$  is at least $\floor{{s}/{2}}^{2}$.
\end{lemma}

\noindent We obtain the following observation from Lemma \ref{HTR}.
\begin{observation} 
\label{obs:htr}
The number of almost halving triangles  of $S$  is at least $\floor{{s}/{2}}^2$.
\end{observation}

\noindent We consider the following lemma which gives a non-trivial lower bound on the number of $(\leq j)$-facets of $S$.
\begin{lemma}\cite{ALCHO}
\label{ALCH}
For $j < {s}/{4}$ , the number of $(\leq j)$-facets of $S$  is at least $4\dbinom{j+3}{3}$.
\end{lemma}

\noindent{\textbf {$k$-Set:}} \cite[page 265]{JM}  A {\textit{$k$-set}} of $S$ is a non-empty subset of $S$ having size $k$ that can be separated from the rest of the points by a $2$-dimensional hyperplane that does not pass through any of the points in $S$.
\vspace{0.2cm}\\
 Let us denote the number of $k$-sets of $S$ by $e_{k}$.\\

\noindent{\textbf {$(\leq k)$-Set:}} \cite{ANDR} A subset $T \subseteq S$ is called an {\textit{$(\leq k)$-set}} if $1 \leq |T| \leq k$ and $T$ can be separated from $S \setminus T$ by a $2$-dimensional hyperplane that does not pass through any of the points in $S$. 
\vspace{0.2cm}\\
Andrzejak et. al. \cite{ANDR} gave the following lemma about the relation between the $j$-facets and the $k$-sets of $S$.

\begin{lemma}\cite{ANDR}
\label{AND}
$e_1 = (E_0/2)+2$, $e_{s-1} = (E_{s-3}/2)+2$, and $e_{k} = (E_{k-1}+E_{k-2})/2+2$ for each $k$ in the range $2 \leq k \leq s-2$.
\end{lemma}

\noindent We obtain the following observation from Observation $2$ and Lemma \ref{AND}.
\begin{observation}
\label{obs101}
There exist  $\Omega(s^2)$  $k$-sets of $S$ such that $min\{k,s-k\}$  is at least $\ceil{(s-1)/2}$.
\end{observation}

\noindent We obtain the following observation from Lemma \ref{ALCH} and  Lemma \ref{AND}.
\begin{observation}
\label{obs100}
The number of $\left(\leq \ceil{{s}/{4}}\right)$-sets of $S$  is  $\Omega(s^3)$.
\end{observation}

\vspace{-0.5cm}
\section{Improved Lower Bound on  $\overline {cr}_d(K_{2d}^d)$}
\label{imprvcross}

In this section, we first use Observation $1$ to improve the lower bound on  $\overline {cr}_d(K_{2d}^d)$ to  $\Omega (2^d \sqrt{d})$. We then present the proofs of Theorems \ref{thm1}, \ref{thm2} and \ref{thm3}. Note that  $V=\{v_1, v_2, \ldots, v_{2d}\}$ denotes the set of points corresponding to the vertices in a $d$-dimensional rectilinear drawing of  $K_{2d}^d$ and  $E$ denotes the set of $(d-1)$-simplices created by the corresponding hyperedges.\\

\noindent{\textbf{Proof of Theorem \ref{thm4}:}} Consider a set $V'=\{v_1, v_2, \ldots, v_{d+4}\} \subset V$, whose Gale transform $D(V')$ is a set of $d+4$ vectors in $\mathbb{R}^3$. As mentioned before, the vectors in $D(V')$ can be treated as points in $\mathbb{R}^3$. In order to apply the Ham-Sandwich theorem to obtain a proper linear separation of $D(V')$, we keep the origin in a set and all the points in $D(V')$ in another set. The Ham-Sandwich theorem implies that there exists a linear hyperplane $h$ such that each of the open half-spaces created by it contains at most $\floor{{(d+4)}/{2}}$ vectors of $D(V')$. Since the vectors in $D(V')$ are in general position in $\mathbb{R}^3$, note that at most two vectors in $D(V')$ can lie on $h$ and no two vectors in $D(V')$ lie on a line passing through the origin. As a result, it can be easily seen that we can slightly rotate $h$  to obtain a linear hyperplane $h'$ which creates a proper linear separation of $D(V')$. Consider a hyperplane parallel to $h'$ and project the vectors in $D(V')$ on this hyperplane to obtain an affine Gale diagram $\overline{D(V')}$. Note that $\overline{D(V')}$ contains $\floor{{(d+4)}/{2}}$ points of the same  color and $\ceil{{(d+4)}/{2}}$ points of the other color in $\mathbb{R}^2$. Without loss of generality, let us assume that the majority color is white. Also, note that the points in $\overline{D(V')}$ are in general position in $\mathbb{R}^2$. \par
 Observation $1$ implies that there exist at least $\floor{{(d+4)}/{2}}$ almost balanced directed lines of  $\overline{D(V')}$. Consider an almost balanced directed line that passes through a white and a black point in $\overline{D(V')}$. Consider the middle point $p$ of the straight line segment connecting  these two points. We  rotate the almost balanced directed line slightly counter-clockwise around $p$ to obtain a partition of $\overline{D(V')}$ by a directed line that does not pass through any point of $\overline{D(V')}$. Note that this partition of $\overline{D(V')}$ corresponds to a unique linear separation of  $D(V')$ having at least $\floor{{(d+2)}/{2}}$ vectors in each of the open half-spaces created by the corresponding linear hyperplane. This implies that there exist at least $\floor{{(d+4)}/{2}}$ distinct linear separations of $D(V')$ such that each such linear separation contains at least $\floor{{(d+2)}/{2}}$ vectors in each of the open half-spaces created by the corresponding linear hyperplane.  Lemma \ref{bjection} implies that there exists a unique crossing pair of $u$-simplex and $v$-simplex corresponding to each linear separation of $D(V')$, such that $u+v+2=d+4$ and $min\{u+1,v+1\} \geq \floor{{(d+2)}/{2}}$. It follows from Lemma \ref{extension} that each such crossing pair of $u$-simplex and $v$-simplex can be extended to obtain at least $\dbinom{d-4}{d-\floor{(d+2)/2}}= \Omega\left({2^d}/{\sqrt{d}}\right)$ crossing pairs of $(d-1)$-simplices formed by the hyperedges in $E$.  Therefore, the total number of crossing pairs of hyperedges in a $d$-dimensional rectilinear drawing of $K_{2d}^d$ is at least $\floor{{(d+4)}/{2}}\Omega\left({2^d}/{\sqrt{d}}\right)=\Omega \left(2^d \sqrt{d}\right)$. \qed 
\vspace{0.3cm}

\noindent{\textbf{Proof of Theorem \ref{thm1}:}}
 Since the points in $V$ are not in  convex position in $\mathbb{R}^d$, we assume  without loss of generality that $v_{d+2}$ can be expressed as a convex combination of the points in $V \setminus \{v_{d+2}\}$. The Carath\'{e}odory's theorem implies that $v_{d+2}$  can be expressed as a convex combination of $d+1$ points in $V \setminus \{v_{d+2}\}$. Without loss of generality, we assume these $d+1$ points to be $\{v_1, v_2, \ldots, v_{d+1}\}$. \par
Consider the set of points $V'=\{v_1, v_2, \ldots,  v_{d+5}\} \subset V$. Note that a Gale transform $D(V')$ of it is a collection of $d+5$ vectors in $\mathbb{R}^4$. Lemma \ref{convexity} implies that there exists a linear hyperplane $h$ that partitions $D(V')$ in such a way that one of the open half-spaces created by $h$ contains exactly one vector of $D(V')$. Since the points in $V'$ are in general position in $\mathbb{R}^4$, Lemma \ref{genposi} implies that at most three vectors of $D(V')$  lie on $h$. Since the vectors in $D(V')$ are in general position, it can be easily seen that we can slightly rotate $h$  to obtain a linear hyperplane $h'$ that partitions $D(V')$ such that one of the open half-spaces created by $h'$ contains $d+4$ vectors and the other one contains exactly one vector. \par
 Consider a hyperplane parallel to $h'$.  We project the vectors in $D(V')$ on this hyperplane to obtain an affine Gale diagram $\overline{D(V')}$. Note that $\overline{D(V')}$ contains $d+4$ points of the same  color and one point of the other color in $\mathbb{R}^3$. Without loss of generality, let us assume that the majority color is white. Also, note that the points in $\overline{D(V')}$ are in general position in $\mathbb{R}^3$ since the corresponding vectors in the Gale transform $D(V')$ are in general position in $\mathbb{R}^4$.\par
 Consider the set $W$ containing $d+4$ white points of $\overline{D(V')}$ in $\mathbb{R}^3$.  Observation $3$ implies that there exist  $\Omega(d^2)$ distinct $k$-sets of $W$ such that $min\{k,d+4-k\}$  is at least $\ceil{{(d+3)}/{2}}$. Each of these $k$-sets corresponds to a unique linear separation of  $D(V')$ having at least $\ceil{{(d+3)}/{2}}$ vectors in each of the open half-spaces created by the corresponding linear hyperplane. Lemma \ref{bjection} implies that there exists a unique crossing pair of $u$-simplex and $v$-simplex corresponding to each of these linear separations of  $D(V')$, such that $u+v+2=d+5$ and $min\{u+1,v+1\} \geq \ceil{({d+3})/{2}}$. It follows from Lemma \ref{extension} that each such crossing pair of $u$-simplex and $v$-simplex can be extended to obtain at least $\test{d-5}{d-\ceil{(d+3)/{2}}}$ crossing pairs of $(d-1)$-simplices formed by the hyperedges in $E$. Therefore, the total number of crossing pairs of hyperedges in such a $d$-dimensional rectilinear drawing of $K_{2d}^d$  is at least $\Omega(d^2)\test{d-5}{d-\ceil{({d+3})/{2}}}=\Omega\left(2^d{d}^{3/2}\right)$. \qed
\vspace{0.3cm}

\noindent{\textbf{Proof of Theorem \ref{thm2}:}}
Consider the points in $V$ that form the vertex set of a $d$-dimensional $t$-neighborly polytope which is not $(t+1)$-neighborly. Lemma \ref{neigh} implies that there exists a linear hyperplane $\widetilde{h}$ such that one of the open half-spaces created by it contains $t+1$ vectors of $D(V)$. Without loss of generality, we denote the set of these $t+1$ vectors by $D^+(V)$.  It implies that one of the closed half-spaces created by $\widetilde{h}$ contains  $2d-t-1$ vectors of $D(V)$. If $d-2$ vectors of $D(V)$ do not lie on $\widetilde{h}$, we rotate $\widetilde{h}$ around the lower dimensional hyperplane spanned by the vectors on $\widetilde{h}$ till some new vector $g_i \in D(V)$ lies on it. We keep rotating $\widetilde{h}$ in this way till $d-2$ vectors of D(V) lie on it.  Lemma \ref{neigh} implies that none of these $d-2$ vectors belongs to the set $D^+(V)$.  After rotating $\widetilde{h}$ in the above mentioned way, we obtain a partition of  $D(V)$ by a linear hyperplane $\widetilde{h'}$ such that one of the open half-spaces created by it contains $t+1$ vectors and the other one contains $d+1-t$ vectors.  This implies that there exist a $t$-simplex and a $(d-t)$-simplex created by the vertices in $V$ such that they form a crossing. We choose any three vertices from the rest of  the $d-2$ vertices in $V$ and add these to the $(d-t)$-simplex to form a $(d+3-t)$-simplex. Lemma \ref{extension} implies that  the $t$-simplex forms a crossing with this $(d+3-t)$-simplex. This implies that the  $t$-neighborly sub-polytope formed by the convex hull of the $d+5$ vertices corresponding to these two simplices is  not $(t+1)$-neighborly.\\
\indent ~~Without loss of generality, let the vertex set of this sub-polytope be $V'=\{v_1, v_2,$  $\ldots, v_{d+5}\}$. Note that a Gale transform $D(V')$ of it is a collection of $d+5$ vectors in $\mathbb{R}^4$. Lemma \ref{neigh} implies that there exists a  linear hyperplane $h$ such that one of the open half-spaces created by it contains exactly $t+1$ vectors of $D(V')$. As described in the proof of Theorem \ref{thm1}, it follows from Lemma \ref{genposi} that at most three vectors can lie on $h$. Since the vectors in $D(V')$ are in general position, we can slightly rotate $h$ to obtain a linear hyperplane $h'$ such that one of the open half-spaces created by $h'$ contains $t+1$ vectors and the other one contains $d+4-t$ vectors.\\
 \indent Consider a hyperplane parallel to $h'$ and project the vectors in $D(V')$ on this hyperplane to obtain an affine Gale diagram $\overline{D(V')}$. Note that $\overline{D(V')}$ contains $d+4-t$ points of the same  color and $t+1$ points of the other color in $\mathbb{R}^3$. Without loss of generality, let us assume that these $d+4-t$ points of the same color are white. Also, note that the points in $\overline{D(V')}$ are in general position in $\mathbb{R}^3$.\\
\indent Let us consider the set $W$ consisting of $d+4-t$ white points of $\overline{D(V')}$. Observation $3$ implies that there exist  $\Omega(d^2)$ distinct $k$-sets of $W$ such that $min\{k,d+4-t-k\}$  is at least $\ceil{({d+3-t})/{2}}$. Each of these $k$-sets corresponds to a unique linear separation of  $D(V')$ such that it contains at least $\ceil{{(d+3-t)}/{2}}$ vectors in each of the open half-spaces created by the corresponding linear hyperplane. Lemma \ref{bjection} implies that there exists a unique crossing pair of $u$-simplex and $v$-simplex corresponding to each of these linear separations of  $D(V')$, such that $u+v+2=d+5$ and $min\{u+1,v+1\} \geq$ $ \ceil{{(d+3-t)}/{2}}$. It follows from Lemma \ref{extension} that each such crossing pair of $u$-simplex and $v$-simplex can be extended to obtain at least $\test{d-5}{d-\ceil{{(d+3-t)}/{2}}}$ crossing pairs of $(d-1)$-simplices formed by the hyperedges in $E$. Therefore, the total number of crossing pairs of hyperedges in such a $d$-dimensional rectilinear drawing of $K_{2d}^d$  is at least $\Omega(d^2)\test{d-5}{d-\ceil{{(d+3-t)}/{2}}}=\Omega\left(2^d{d}^{3/2}\right)$. \qed
\vspace{0.34cm}

\noindent{\textbf{Proof of Theorem \ref{thm3}:}}
Since the points in $V$ form the vertex set of a $d$-dimensional $(\floor{d/2}-t')$-neighborly polytope, consider a sub-polytope of it formed by the convex hull of the vertex set $V'$ containing any d+5 points of $V$. Without loss of generality, let   $V'$  be $\{v_1, v_2, \ldots, v_{d+5}\}$. Note that a Gale transform $D(V')$ of it is a collection of $d+5$ vectors in $\mathbb{R}^4$ and an affine Gale diagram $\overline{D(V')}$ of it  is a collection of $d+5$ points in $\mathbb{R}^3$. In this proof, we ignore the colors of these points. However, note that the points in $\overline{D(V')}$ are in general position in $\mathbb{R}^3$.\par
 Consider the set $\overline{D(V')}$.  It follows from Observation $4$ that the number of $\left(\leq\ceil{{(d+5)}/{4}}\right)$-sets of $\overline{D(V')}$ is $\Omega(d^3)$.  For each $k$ in the range $ 1\leq k \leq$ ${\scriptstyle \ceil{{(d+5)}/{4}}}$, a $k$-set of $\overline{D(V')}$  corresponds to a unique linear separation of  $D(V')$.  Lemma \ref{neigh} implies that each of these $\Omega(d^3)$ linear separations of $D(V')$ contains at least $\floor{d/2}-t'+1$ vectors in each of the open half-spaces created by the corresponding linear hyperplane. Lemma \ref{bjection} implies that there exists a unique crossing pair of $u$-simplex and $v$-simplex corresponding to each linear separation of $D(V')$, such that $u+v+2=d+5$ and $min\{u+1,v+1\} \geq \floor{d/2}-t'+1$. It follows from Lemma \ref{extension} that each such crossing pair of $u$-simplex and $v$-simplex can be extended to obtain at least $\dbinom{d-5}{d-\floor{d/2}+t'-1}= \Omega\left({2^d}/{\sqrt{d}}\right)$ crossing pairs of $(d-1)$-simplices formed by the hyperedges in $E$.  Therefore, the total number of crossing pairs of hyperedges in such a $d$-dimensional rectilinear drawing of $K_{2d}^d$ is  $\Omega(d^3)\Omega\left({2^d}/{\sqrt{d}}\right)=\Omega \left(2^d d^{5/2}\right)$. \qed

\end{document}